\documentclass[reqno,a4paper,12pt]{amsart}
\usepackage{amsmath, bbm, amsthm,amsfonts,amssymb, mathrsfs, amsxtra, ae}
\usepackage{verbatim, graphicx, ifthen, enumitem}
\usepackage[T1]{fontenc}
\usepackage [applemac] {inputenc}
\usepackage{enumerate}

\allowdisplaybreaks

\let\oldmarginpar\marginpar
\renewcommand\marginpar[1]{\-\oldmarginpar[\raggedleft\footnotesize #1]%
{\raggedright\footnotesize #1}}
\newtheorem{theorem}{Theorem}

\newtheorem{corollary}{Corollary}
\theoremstyle{definition}

\newtheorem{question}{Question}

\theoremstyle{remark}

\newcommand{\Z}{\mathbb{Z}}

\newcommand{\abs}[1]{| #1 |}
\newcommand{\Abs}[1]{\left| #1 \right|}
\newcommand{\Bigabs}[1]{\Big| #1 \Big|}
\newcommand{\Norm}[1]{\left\| #1 \right\|}
\newcommand{\Bignorm}[1]{\Big\| #1 \Big\|}
\newcommand{\inner}[2]{\left\langle #1|#2 \right\rangle}

\newcommand{\N}{\mathbb{N}}
\newcommand{\R}{\mathbb{R}}
\newcommand{\T}{\mathbb{T}}
\newcommand{\C}{\mathbb{C}}

\newcommand{\Q}{\mathbb{Q}}

\newcommand{\Hp}{\mathscr{H}}

\def\T{\mathbb{T}}
\def\N{\mathbb{N}}
\def\Z{\mathbb{Z}}
\def\Q{\mathbb{Q}}
\def\R{\mathbb{R}}
\def\C{\mathbb{C}}

\def\1{\mathbf{1}}

\newcommand{\dif}{\mathrm{d}}
\newcommand{\e}{\mathrm{e}}
\newcommand{\im}{\mathrm{i}}
\newcommand{\norm}[1]{\|#1\|}

\renewcommand{\Re}{\operatorname{Re}}

\newcommand{\beqno}{\begin{eqnarray*}}
\newcommand{\eeqno}{\end{eqnarray*}}
\newcommand{\beqla}[1] {\begin {eqnarray}\label{#1}}
\def\eeq {\end {eqnarray}}
\newcommand{\beq}{\begin {eqnarray}}

\newcommand{\real}{{\mathbb R}}

\newcommand{\complex}{{\mathbb C}}

\newcommand{\sgn}{{\rm sgn}\,}

\begin{document}

 \title[Fourier multipliers for Dirichlet series]{Fourier multipliers for Hardy spaces \\ of Dirichlet series}
 
\author{Alexandru Aleman, Jan-Fredrik Olsen \\ and {Eero Saksman}}
\address{Centre for Mathematical Sciences, Lund University, P.O. Box 118, SE-22100 Lund, Sweden}
\email{aleman@maths.lth.se}
\address{Department of Mathematics and Statistics, University of Helsinki, P.O. Box 68, FI-00014 Helsinki, Finland}
\email{eero.saksman@helsinki.fi}
\address{Centre for Mathematical Sciences, Lund University, P.O. Box 118, SE-22100 Lund, Sweden}
\email{janfreol@maths.lth.se}
\thanks{The third author was
supported by the Finnish CoE in Analysis and Dynamics Research,
and by the Academy of Finland, projects 
113826 \& 118765}

\begin{abstract}
	We obtain  new results on Fourier multipliers for Dirichlet-Hardy spaces. As a consequence, we establish a  Littlewood-Paley  type inequality which yields  a simple  proof that the Dirichlet monomials form a Schauder basis for $p>1$. 
\end{abstract}

\keywords{Dirichlet series, Hardy spaces, infinite polydisc, Schauder bases, Fourier multipliers}
\subjclass[2000]{Primary 30B50; Secondary 42B15, 42B30, 46B15}

\maketitle

\section{Introduction}
The Dirichlet-Hardy spaces $\Hp^p$ were first explicitly studied in the papers \cite{bayart2002paper, hls1997}. (We refer to these papers for full details of the discussion in this section. See also \cite{hls1999} for some historical remarks.) For $p=2$, they   consist of   Dirichlet series  $\sum_{n \in \N} a_n n^{-s}$ with square-summable coefficients, where $s = \sigma + \im t$ denotes the complex variable. By the Cauchy-Schwarz inequality, functions in $\Hp^2$ converge on the half-plane $\C_{1/2}   =  \{ \sigma > 1/2\}$. 
These spaces connect function space theory   to analytic number theory. A striking illustration of this connection is given by the Riemann-zeta function $\zeta(s) = \sum_{n \in \N} n^{-s}$ that   gives the reproducing kernel of $\Hp^2$. Indeed, the function $k_w(s) := \zeta(s + \bar{w})$, for $\Re w >1/2$, has the property that $\inner{f}{k_w} = f(w)$ for all $f \in \Hp^2$, as may be  verified by inspection. 

For general $p>0$, these spaces are defined to be the closure of Dirichlet polynomials in the norm
\begin{equation} \label{ergodic norm}
	  \lim_{T \rightarrow \infty}  \left( \frac{1}{2T} \int_{-T}^T \Bigabs{\sum_{n=1}^N a_n n^{- \im t}}^p \dif t \right)^{1/p}.
\end{equation}
This norm can be understood as the ergodic theorem on the infinite dimensional polydisk $\T^\infty$.
To briefly explain this, we note that $\T^\infty$ is a compact Abelian group with the product of the normalised Lebesgue measures $\dif \theta_i/2\pi$ on each copy of $\T$ as its unique normalised Haar measure $\dif \theta$. It has dual group $\Z^\infty_\text{fin}$, i.e., sequences in $\Z^\infty$ with finitely many non-zero coefficients. So by standard Fourier analysis on groups,   $F \in L^p(\T^\infty)$ has a Fourier expansion $F \sim \sum_{\nu \in \Z^\infty_\text{fin}}a_\nu z^\nu$, where $z \in \T^\infty$ and we use multi-index notation.
The central observation, essentially called Kroenecker's lemma, is that the path $\phi: t \mapsto (2^{-\im t}, \ldots, p_i^{- \im t},\ldots)$, where $p_i$ is the $i$'th prime number, is ergodic in $\T^\infty = \{ z = (z_1, \ldots) : z_i \in \T \}$. The ergodic theorem now says exactly that for continuous functions
\begin{equation}
	\lim_{T \rightarrow \infty} \left( \frac{1}{2T} \int_{-T}^T \abs{ F\circ \phi(t) }^p \dif t\right)^{1/p} = \norm{F}_{L^p(\T^\infty)}
\end{equation}
For $F$ with spectral support only in the narrow cone $\N^\infty_\text{fin}$, one checks that $F\circ \phi$ is a Dirichlet series and that the right-hand side of this formula is exactly \eqref{ergodic norm}, provided we identify $a_\nu$ with $a_n$ when $n = p_1^{\nu_1} p_2^{\nu_2} \cdots$. (Note that the same argument can be made using only the Stone-Weierstrass theorem, see \cite{saksman_seip2009}) We define the subspace $H^p(\T^\infty)$ to consist of exactly these functions. By the uniqueness of prime number factorization, the map from $H^p(\T^\infty)$ to $\Hp^p$ given by $F \mapsto F \circ \phi$ has an inverse, which is called the Bohr lift in honor of H. Bohr.

The structure of the paper is as follows.
In Section \ref{multiplier section}, we use a technique of Fefferman to study certain Fourier multipliers on the spaces $L^p(\T^\infty)$. These results are used in Section \ref{consequence section} to obtain a Littlewood-Paley inequality for the spaces $\Hp^p$:
	for $f = \sum_{n \in \N} a_n n^{-s}$   in $\Hp^p$ with $p> 1$ and $c>1$, we have
	\begin{equation} \label{intro ineq}
		\norm{f}_{\Hp^p} \simeq \abs{a_0} + \Norm{ \left(\sum_{k \geq 0} \abs{\sum_{\log n \in (c^k, c^{k+1})} a_n n^{-s}}^2\right)^{1/2} }_{\Hp^p}.
	\end{equation}
%
As an application,  we observe that the functions $\{n^{-s}\}_{n \in \N}$ constitute a Schauder basis for the spaces $\Hp^p$ for $p>1$.

\section{Fourier Multipliers} \label{multiplier section}
To   state and prove our theorem on Fourier multipliers, we first introduce some notation, and review some necessary background. Throughout the section, $p \geq 1$.

A measurable function $m : \R \rightarrow \C$ is called a Fourier multiplier on $L^p(\R)$ if the operator
$f \longmapsto \mathcal{F}^{-1}( m(\xi) \hat{f}(\xi))$ is bounded
 on $L^p(\R)$, where $\mathcal{F}$ denotes the Fourier transform. On the torus $\T$, a function $m : \Z \rightarrow \C$ is called a multiplier if the map
defined by the relation $\e^{\im n t} \mapsto m(n) \e^{\im n t}$
extends to a bounded operator on $L^p(\T)$.  Finally, a function $m : \Z^\infty_{\text{fin}} \rightarrow \C$ is called a multiplier if the operator
\begin{equation*}
	 \sum_{\nu \in \Z^\infty_{\text{fin}}} a_\nu \e^{\im \nu \cdot \theta} \longmapsto \sum_{\nu \in \Z^\infty_{\text{fin}}} m(\nu) a_\nu \e^{\im \nu \cdot \theta}
\end{equation*}
is bounded on $L^p(\T^\infty)$. Here we use the notation $z=\e^{\im \theta}$ for a point in $\T^\infty$.
We denote the respective operator norms by $\norm{m}_{M_p(X)}$, where $X= \R, \T$ or $\T^\infty$ as appropriate. We refer to the operator of multiplication by $m$ by $T_m$.

It is well-known that results for multipliers on $\T$ may be deduced from those on the real line by the method of transference. More specifically, let $m : \R \rightarrow \C$ be a regulated function, i.e.,
\begin{equation*}
	m(\xi ) = \lim_{\epsilon \rightarrow 0^+} \frac{1}{2\epsilon} \int_{-\epsilon}^\epsilon m(\xi + t) \dif t, \qquad \forall \xi \in \R. 
\end{equation*}
The basic result on transference, due to de Leeuw \cite{deleeuw1965} (see \cite[Section 3.6.2]{grafakos2008a} for proofs), states that if a regulated function $m$ is a multiplier on $\R$,
then $m$ restricted to $\Z$ is a multiplier on the torus. A converse statement also holds. In fact,
\begin{equation} \label{transference}
	\norm{m}_{M_p(\R)}  = \sup_{\gamma > 0} \norm{m(\gamma \cdot)}_{M_p(\T)}.
\end{equation}
Our argument relies in a crucial way on this formula.

To formulate our result, we introduce some additional notation. For $\nu \in \Z^\infty_{\text{fin}}$, one associates a unique rational number:
\begin{equation*}
	r : \nu \longmapsto r_\nu   = p_1^{\nu_1} \cdots p_k^{\nu_k},
\end{equation*}
where $p_i$ is the $i$'th prime number.
So, given a function $m : \Q_+ \rightarrow \C$, we obtain a function $m \circ r : \Z^\infty_{\text{fin}} \rightarrow \C$. In particular, $m$ induces in this way
a densely defined Fourier multiplier on $L^p(\T^\infty)$ by 
\begin{equation*}
	 T_{m \circ r} : \sum_{\nu \in \Z^\infty_{\text{fin}}} a_\nu \e^{\im \nu \cdot \theta} \longmapsto \sum_{\nu \in \Z^\infty_{\text{fin}}} m( r_\nu) a_\nu \e^{\im \nu \cdot \theta}
\end{equation*}

Our multiplier result is as follows:

\begin{theorem} \label{multiplier theorem}
	Let $p \in [1,\infty)$ and $m : \R_+ \rightarrow \C$ be a regulated function continuous at rational points. Then $m \circ r$ is a   Fourier multiplier on $L^p(\T^\infty)$, where $r(\nu) = p_1^{\nu_1}\cdots p_k^{\nu_k}$ for $\nu \in \Z^\infty_{\mathrm{fin}}$, if and only if $m \circ \exp$ is
	a   Fourier multiplier on $L^p(\R)$.
	Moreover, 
	\begin{equation*}
		\norm{m 	 \circ r}_{M_p(\T^\infty)} = \norm{m \circ \exp}_{M_p(\R)}.
	\end{equation*}
\end{theorem}
\begin{proof}
	We split the proof of the theorem into two parts.

	First, we establish that $\norm{m  \circ  r}_{M_p(\T^\infty)} \leq \norm{m \circ \exp}_{M_p(\R)}$.  Fix a polynomial
	\begin{equation*}
		f = \sum_{\nu \in \Z^\infty_{\text{fin}}}  a_\nu \e^{\im \nu \cdot \theta}.
	\end{equation*}
	Observe that since a polynomial only depends on a finite number of variables,   we may restrict our attention to $L^p(\T^d)$, for some $d \in \N$. 
	As a multiplier on $L^p(\T^d)$, we need only consider $\nu \in \Z^d$. Explicitly, we only need to consider the multiplier
	\begin{equation*}
		\nu\mapsto m(r_\nu) = m\left( \e^{\nu_1 \log p_1 + \ldots + \nu_d \log p_d} \right),\quad \nu\in\Z^d,
	\end{equation*}
acting on $L^p(\T^d).$
	The idea is to introduce a change of variables on $\T^d$ so that as a multiplier, this function only acts on the first variable.
	
	To do this, we need to make an approximation. For $\delta >0$, choose $Q, a_1, \ldots, a_d \in \N$ so that
	\begin{equation*}
		\Abs{\frac{a_j}{Q}  -   \log p_j  }< \delta, \qquad \text{for} \; j = 1, \ldots, d. 
	\end{equation*}
	We may assume that $a_1$ and $a_2$ are relatively prime (indeed, by   the prime number theorem, we may choose both $a_1$ and $a_2$ to be prime), whence there exist $q_1, q_2 \in \N$ so that $a_1 q_2 - a_2 q_1  = 1$. 
	This ensures that the $d\times d$ matrix
	\begin{equation*}
		A = \left( \begin{matrix} a_1 & a_2 & a_3 & \cdots & a_d \\ q_1 & q_2  & 0  & \cdots & 0 \\ 0  &  0 & 1 & \cdots  & 0 \\ \vdots & \vdots & \vdots & \ddots & \vdots
		\\ 0 & 0 & 0 & \cdots & 1  \end{matrix} \right) 
	\end{equation*}
	satisfies $\det A = 1$. A fortiori,   $A^{-1}$ also has integer coefficients, whence $A:\Z^d\to\Z^d$ is bijective.  Especially, one checks that $A$ induces   a bijective and measure preserving diffeomorphism on $\T^d=\R^d/\Z^d$.
	
	We next introduce a function defined on  $\nu \in \Z^d$ by
	\begin{equation*}
		M(\nu)  = m\left(\e^{\frac{a_1}{Q} \nu_1 + \cdots + \frac{a_d}{Q} \nu_d} \right).
	\end{equation*}
	Since $m$ was assumed to be continuous on rational numbers, it follows that for any $\epsilon >0$ small enough, we may choose $\delta >0$ 
	sufficiently small in the above approximation, so that $\abs{M(\nu) - m(r_\nu)} < \epsilon$ uniformly on the finite index set corresponding to
	the set of non-zero coefficients of the polynomial $f$. In particular, this implies that we can make $\norm{T_M f  - T_{m \circ r} f}_{L^p(\T^d)}$ arbitrarily small.
	In light of \eqref{transference}, to obtain the desired inequality, we infer that it suffices to prove 
	\begin{equation} \label{sufficiency condition}
		\norm{T_M f}_{L^p(\T^d)}  \leq \norm{m \circ \exp(Q^{-1} \cdot)}_{M_p(\T)}.
	\end{equation}

	To verify  \eqref{sufficiency condition}, let us first employ the change of variables $\theta = A^T \theta'$ to get
	\begin{align*}
		\norm{T_M f }^p_{L^p(\T^d)} 
		= \int_{\T^d} \Bigabs{ \sum_{\nu \in \Z^d}  M(\nu) a_\nu \e^{\im \nu \cdot A^T \theta'}  }^p \dif \theta' 
		=\int_{\T^d} \Bigabs{ \sum_{\nu \in \Z^d} M(\nu)  a_\nu \e^{\im A\nu \cdot  \theta'}  }^p \dif \theta'.
	\end{align*}
	If we change indices by $\nu' = A \nu$, and observe that $M(\nu) = m(\e^{\nu'_1/Q})$, this becomes
	\begin{equation} \label{step to be modified}
		\int_{\T^d} \Bigabs{ \sum_{\nu' \in \Z^d}  m(\e^{\nu'_1/Q}) a_{A^{-1}\nu'} \e^{\im \nu' \cdot  \theta'}  }^p \dif \theta'
		=
		\int_{\T^{d-1}} \Bignorm{ \sum_{\nu'_1} b_{\nu'_1} m(\e^{\nu'_1/Q}) \e^{\im \nu'_1 \theta'_1} }_{L^p(d\theta'_1)}^p \frac{\dif \theta'_2}{2\pi} \cdots \frac{\dif \theta'_d}{2\pi},
	\end{equation}
 	where $b_{\nu'_1}=b_{\nu'_1}(\theta'_2,\ldots, \theta'_d) $ is constant with respect to $\theta'_1$. This is less than or equal to
	\begin{multline*}
		\norm{m\circ \exp (Q^{-1} \cdot)}_{M_p(\T)}^p \int_{\T^{d-1}} \Bignorm{ \sum_{\nu'_1} b_{\nu'_1} \e^{\im \nu'_1 \theta'_1} }_{L^p(d\theta'_1)}^p \frac{\dif \theta'_2}{2\pi} \cdots \frac{\dif \theta'_d}{2\pi} \\
		 = \norm{m\circ \exp (Q^{-1} \cdot)}_{M_p(\T)}^p  \norm{f}_{L^p(\T^d)}^p,
	\end{multline*}
	which exactly yields the desired inequality \eqref{sufficiency condition}.
 
	We turn to the second part of the proof, where we     establish the inequality $ \norm{m \circ \exp}_{M_p(\R)} \leq \norm{m\circ r}_{M_p(\T^\infty)}$.
	 By \eqref{transference}, it is sufficient to show that, for every $\gamma >0$, we have
	$\norm{m \circ \exp(\gamma \cdot)}_{M_p(\T)} \leq \norm{m \circ r}_{M_p(\T^\infty)}$.
	We now fix
	a polynomial in one variable. As   our idea is to work the previous argument backwards using only two variables, we express the polynomial as trivially depending on a second variable:
	\begin{equation*}
		f(\theta'_1, \theta'_2) =  \sum_{\abs{n} \leq N} a_{(n,0)} \e^{\im n \theta'_1}.
	\end{equation*}
	Here $a_{(n,m)}$ is zero for all $(n,m) \notin \N \times \{0\}$.

	As in the first part of the proof, we first fix $\delta >0$ and introduce a change of variables, this time induced by the matrix
	\begin{equation*}
		B = \left( \begin{matrix} b+1 & b \\ 1 & 1 \end{matrix} \right).
	\end{equation*}
Above, the integer  $b$ is chosen so large that there exist  prime numbers 
$p_j, p_{k}$ for which
	\begin{equation*}
		\abs{\gamma (b+1) - \log p_j} < \delta/N, \qquad \text{and} \qquad 		\abs{\gamma b - \log p_k} < \delta/N.
	\end{equation*}
	This is possible, since from the prime number theorem   it   holds that
$
 \log ({p_{n+1}}/{p_n}) \rightarrow 0$ when $n \rightarrow \infty.$
As we have $\det B = 1$, the matrix $B$ induces a measure preserving diffeomorphism of $\T^2$.
	
	Setting $\theta = B^T \theta'$ and $(n,0)^T = B \nu$, we get
	\begin{align*}
		 \norm{T_{m \circ \exp(\gamma \cdot)} f }_{L^p(\T)}^p
		 &=
		 \iint_{\T^2} \Bigabs{ \sum_{\abs{n} \leq N}m(\e^{\gamma n})  a_{(n, 0)}  \e^{\im (n,0) \cdot \theta'} }^p \frac{\dif \theta'_1 }{2\pi}\frac{\dif \theta'_2}{2\pi} \\
		 &=\iint_{\T^2} \Bigabs{ \sum_{\nu \in \Z^2}  m(\exp(\gamma((b+1)\nu_1+b\nu_2)))a_{B\nu} \e^{\im \nu \cdot \theta} }^p \frac{\dif \theta_1 }{2\pi}\frac{\dif \theta_2}{2\pi}.
	\end{align*}

	As $v=(n,-n)^{T}$, we are only summing over $\nu \in \Z^2$ for which $\abs{\nu} \leq 2N$. So given any $\epsilon>0$, by choosing $\delta>0$ small enough,	we make 
	\begin{equation*}
		\abs{ m\circ \exp(\gamma n) - m(p_j^{\nu_1} p_k^{\nu_2} )} = 
		\abs{ m (\e^{\gamma(b_1 \nu_1 + b_2 \nu_2)}) - m(\e^{\nu_1 \log p_j+ \nu_2 \log p_k}) } < \epsilon
	\end{equation*}
	uniformly for indices $\nu$ so that $a_{B\nu}$ is non-zero. This implies that we only need to establish that
	\begin{equation*}
		 \int_{\T^2} \Bigabs{ \sum_{\nu \in \Z^2} m(p_j^{\nu_1} p_k^{\nu_2}) a_{B \nu} \e^{\im \nu \cdot \theta}}^p \frac{\dif \theta_1}{2\pi} \frac{\dif \theta_2}{2\pi}
		 \leq
		 \norm{m \circ r}_{M_p(\T^\infty)}^p \norm{f}_{L^p(\T)}^p.
	\end{equation*}
	But this is readily seen to hold, as the left-hand side may be interpreted as $T_{m \circ r}$ applied to a function $F$ depending on the $j$'th and $k$'th copy of $\T$ in $\T^\infty$, and where $\norm{F}_{L^p(\T^\infty)} = \norm{f}_{L^p(\T)}$ holds by reversing the changes in notation and variables.
	\end{proof}

\section{Some Consequences and open problems} \label{consequence section}

In this section we deduce a Littlewood-Paley inequality from Theorem \ref{multiplier theorem}, and also discuss  Schauder bases for the spaces $\Hp^p$.

First, we observe how a characterisation due to Marcinkiewicz 
is inherited by multipliers of the form discussed in the previous section. To do this, we recall that
the total variation of a complex function $f$ on the interval $(a,b)$ is given by 
\begin{equation*}
	\norm{f}_{\text{BV}(a,b)} = \sup \sum_{n=1}^N \abs{f(x_j) - f(x_{j-1})},
\end{equation*}
where the supremum is taken over all sequences $a = x_0 < x_1 < \cdots < x_n = b$. For fixed $\eta >1$, we also use the notation
\begin{equation*}
	I_k =  \left\{ \begin{aligned} \hspace{0.1 cm}[\e^{\eta^k}, \e^{\eta^{k+1}}] && k &\geq 1. \\  [\e^{-\eta}, \e^\eta] && k &= 0, \\ [\e^{-\eta^{\abs{k}+1}}, \e^{-\eta^{\abs{k}}}]  && k &\leq -1. \end{aligned} \right. 
\end{equation*}
We now get:
\begin{corollary} \label{marcinkiewicz corollary}
	Suppose that $p \in (1,\infty)$ and $\eta>1$, then there exists a constant $C>0$ such that  for all regulated $m : \R_+ \rightarrow \C$ that are continuous at rationals we have
	\begin{equation*}
		\norm{m \circ r}_{M_p(\T^\infty)} \leq C \left( \norm{m}_{L^\infty(0,\infty)} + \sup_{k \in \Z} \norm{m}_{\text{BV}(I_k)} \right).
	\end{equation*}
\end{corollary}
\begin{proof}
	Since $m \circ \exp$ and $m$ have the same sup-norm, and $\norm{m \circ \exp}_{\text{BV}(\eta^k,\eta^{k+1})} = \norm{m}_{\text{BV}(I_k)}$, the Marcinkiewicz bound follows  immediately from its classical counter-part, see \cite[Section 5.2.1]{grafakos2008a}.
\end{proof}

We also formulate a H\" ormander-Mihlin type multiplier theorem for $p=1$, see \cite{mihlin1956, hormander1960} (a proof is also found in \cite[Theorem 5.2.7]{grafakos2008a}).
Recall that  $m:\R\to\C$ satisfies the H\" ormander-Mihlin condition if $m$ is continuous and piecewise differentiable on $\real\setminus\{ 0\}$ with
\begin{equation}\label{eq:hm}
\| m\|_{L^\infty(\R)} +\sup_{x\not=0}|xf'(x)| <\infty.
\end{equation}
If this holds, then $m\in M^p(\R)$ for any $p\in (1,\infty )$. In addition, $m$ defines a multiplier operator that is bounded from $H^1(\R)$ to $ L^1(\real)$ with
 norm  bounded by (\ref{eq:hm}). For our purposes it is useful to observe that (\ref{eq:hm}) remains invariant if  $m$
is replaced by $m(\lambda\cdot)$ for any $\lambda >0.$
\begin{corollary}\label{th:2.2} Assume that $m:(0,\infty)\to\complex$ is continuous and piecewise differentiable. Then
\begin{equation*}
\| m \circ r\|_{H^1 (\T^\infty)\to L^1 (T^\infty)}\leq c\Big(
\| m\|_{L^\infty (0,\infty )} + \sup_{t>\e}| t\log (t) m'(t)|
\Big).
\end{equation*}
\end{corollary}
\begin{proof} 
The condition of $m$ ensures that it can be modified on $(0,e)$  so that it satisfies (\ref{eq:hm}) on $\real .$ Hence $T_{m \circ \exp}:H^1 \to L^1$   is bounded in one variable with the stated bound. The case $p=1$ of the  proof of Theorem \ref{multiplier theorem} now applies without changes. One simply needs to observe that
after the change of variables, the assumption $f\in H^1(\T^\infty)$ implies that $a_{A^{-1}\nu'}=0 $ if $\nu_1<0,$ whence the one-dimensional multiplier 
$m(\e^{\nu'_1/Q})$ is applied only to analytic functions.
\end{proof}

We proceed to obtain a Paley-Littlewood type of theorem for $L^p(\T^\infty)$ as a consequence of
Corollary \ref{marcinkiewicz corollary}.  Fix a rational number $\eta>1$, and consider intervals $I_k$ as above. For $f = \sum_{\nu \in \Z^\infty_{\text{fin}}} a_\nu \e^{\im \nu \cdot \theta}$ in $L^p(\T^\infty)$,  define
the square function
\begin{equation*}
	S(f) = \left( \sum_{k} \Bigabs{f_k(\theta)}^{2}  \right)^{1/2},
\end{equation*}
where
\begin{equation*}
	f_k(\theta) = \sum_{\nu : \; r_\nu\in I_k} a_\nu \e^{\im \nu \cdot  \theta}.
\end{equation*}
The following result is clearly the most interesting in the special case of $\Hp^p$, which we stated as formula \eqref{intro ineq} in the introduction.
\begin{corollary} \label{paley-littlewood corollary}
	Suppose that $p \in (1,\infty)$, and that $\eta>1$ is a rational number. Then there exist constants such that for all $f  \in L^p(\T^\infty)$, we have
	\begin{equation*}
		\norm{f}_{L^p(\T^\infty)} \simeq \norm{S(f)}_{L^p(\T^\infty)}.
	\end{equation*}
\end{corollary}
\begin{proof}
	We apply a standard argument. Define $m_{\epsilon} = \sum_{k \in \Z} \epsilon_k \chi_{I_k}$ for given $\epsilon \in \{-1,1\}^\Z$. By Corollary  \ref{marcinkiewicz corollary}, there exists some $C>0$, independent of $\epsilon$, such that   $\norm{m_\epsilon \circ r}_{M_p(\T^\infty)} \leq C$. Here, $m_\epsilon$ is 
	made  regulated by defining it appropriately on the endpoints of the intervals $I_k$. This has no effect on the operator $T_{m_\epsilon \circ r}$ as
	the endpoints are irrational. Next, since $T_{m_\epsilon \circ r}T_{m_\epsilon \circ r} = \text{Id}$, we obtain for any $g \in L^p(\T^\infty)$ that
$\norm{T_{m_\epsilon \circ r} g}_{L^p(\T^\infty)} \simeq \norm{g}_{L^p(\T^\infty)}$. This holds uniformly in $\epsilon$. The corollary now follows by averaging over $\epsilon$ and invoking Khintchine's inequality \cite[p. 435]{grafakos2008a}.
\end{proof}

This result should be compared to a Paley-Littlewood inequality  obtained from martingale theory. Indeed,   a function $f \in L^p(\T^\infty)$ may be considered as a martingale $\{f_{(N)}\}$ with respect to the filtration induced by the increasing sequence of $\sigma$-algebras corresponding to the sequence $\{ \T^N \}_{N\in \N}$. The function $f_{(N)}$, also called the conditional expectation, is obtained from $f$ by integrating away all but the $N$ first variables (see, e.g., \cite{hls1997} where these are called the `$N$:te Abschnitt'). A Paley-Littlewood inequality is  now obtained as  a direct corollary of the classical Burkholder's square function inequality \cite{burkholder1966} (see also
\cite[Theorem 5.4.7]{grafakos2008a}).  Set $\Delta_N f = f_{N} - f_{N-1}$. Then
\begin{equation} \label{martingale}
	\norm{f}_{L^p(\T^\infty)} \simeq \norm{(\sum \abs{\Delta_N f}^2)^{1/2}}_{L^p(\T^\infty)}.
\end{equation}
Actually, the same argument that was used to prove Corollary \ref{paley-littlewood corollary} yields \eqref{martingale} without using probability theory (this observation was applied in  \cite{asmar_newberger_saleem2006}).

In the following corollary, we consider the functions $1, 2^{-s},3^{-s},\ldots$. It is clear that they form an orthogonal basis in $\Hp^2$. Luckily, they also yield a natural basis in $\Hp^p$:
\begin{corollary}
	Suppose $p \in (1,\infty)$. Then the functions $n^{-s}$, $n=1,2,\ldots $, form a Schauder basis for $\Hp^p$.
\end{corollary}
\begin{proof}
	By the density and independence of these functions, and standard Schauder basis theory, it suffices to establish that the truncations $\sum_{n=1}^\infty  a_n n^{-s}\mapsto \sum_{n=1}^N a_n n^{-s} $ are bounded on $\Hp^p$, uniformly
with respect to $N$.
	Let $\alpha \in (0, 1/2)$ be an irrational number. According to Corollary \ref{marcinkiewicz corollary}, the indicator functions of the intervals $(0, N +  \alpha)$
	yield uniformly bounded multipliers on $L^p(\T^\infty)$. The result  follows.
\end{proof}
Although we have not been able to find this result stated explicitly in the literature, we indicate how it can be deduced from \cite[Theorem 8.7.2]{rudin1962}.
This  result deals with the space $L^p(G)$, where $G$ is a compact abelian group that has a dual $\Gamma$ which admits an order relation $P$. I.e., $P$ is a subset of $\Gamma$ such that 
\begin{equation*}
	 P \cup (-P)  = \Gamma, \qquad \text{and} \qquad P \cap (-P) = \{ 0\}.
\end{equation*}
 Under any such order relation one can define $\sgn(\gamma) \in \{ -1, 0, 1\}$ according to whether or not $\gamma$ is in $P$ or is in $\{0\}$. With this, the statement is that  the Hilbert transform
\begin{equation*}
T_P:	\sum_{\gamma \in \Gamma} a_\gamma e_\gamma \longmapsto - \im \sum_{\gamma \in \Gamma} \sgn(\gamma) e_\gamma
\end{equation*}
is bounded on $L^p(G)$, where $e_\gamma$ is the Fourier character corresponding to $\gamma \in \Gamma$. In particular, 
$P = \{ \nu: \log r_\nu  \leq 0 \}$
is an order relation in the dual $\Z^\infty_\text{fin}$ of $\T^\infty$. Hence the  corresponding Riesz projection $R_P$, where $R_Pe_\gamma:= \chi_{\{r_\gamma \geq 0\}}e_\gamma$,  is bounded on $L^p(\T^\infty )$.  If $r\to \nu(r)$ is the inverse of the map $r$, we obtain   uniformly in $N$
\begin{equation*}
	\Bignorm{ \sum_{ n \leq  N} a_n n^{-s}}_{\Hp^p} 
=\Bignorm{e_{\nu (N)}R_P(e_{-\nu (N)}f)}
\end{equation*}
for functions $f(s)= \sum_{ n \in\N} a_n n^{-s}$ in $\Hp^p$. As above, it follows immediately that $\{ n^{-s} \}$ is a Schauder basis for $\Hp^p$ when $p>1$.

We end with the following open questions which may seem innocent, but they could be somewhat hard taking into account the quite intractable and mysterious nature of the spaces $\Hp^p$ for $p\not=2$ as discussed, e.g., in \cite{saksman_seip2009}.

\begin{question}
Does $\Hp^1$ have a Schauder basis? 
\end{question}
\begin{question}
Does $\Hp^p$ have an unconditional basis if $p\in (1,\infty )\setminus \{ 2\}$?
\end{question}

\section*{Acknowledgements}
This work was done as part of the research program ``Complex Analysis and Spectral Problems'' 2010/2011 at the Centre de Recerca
Matem\`atica (CRM), Bellaterra, Barcelona. We would also like to thank Anders Olofsson for valuable discussions on the topic.

\bibliographystyle{amsplain}

\def\cprime{$'$} \def\cprime{$'$} \def\cprime{$'$}
\providecommand{\bysame}{\leavevmode\hbox to3em{\hrulefill}\thinspace}
\providecommand{\MR}{\relax\ifhmode\unskip\space\fi MR }
\providecommand{\MRhref}[2]{%
  \href{http://www.ams.org/mathscinet-getitem?mr=#1}{#2}
}
\providecommand{\href}[2]{#2}

\end{document}